\documentclass[12pt]{smfart}
\usepackage{smfthm}
\usepackage{ucs}
\usepackage[utf8x]{inputenc}
\usepackage[french]{babel}
\usepackage[T1]{fontenc}
\usepackage{lmodern}
\usepackage{amsthm}

\usepackage{amsmath}
\usepackage{amssymb}
\usepackage[all]{xy}
\usepackage{enumerate}
\usepackage{fullpage}

\renewcommand{\b}{F}
\newcommand{\C}{k}

\renewcommand{\hom}{\mathrm{Hom}}
\newcommand{\K}{\mathcal{K}}
\newcommand{\ind}{\mathrm{ind}}

\newcommand{\locart}{\mathrm{Mod}^{la}_A}
\newcommand{\promod}{\mathrm{Mod}^{pro}_A}
\newcommand{\res}{k_F}

\newcommand{\tor}{\mathrm{Tor}}

\numberwithin{equation}{section}

\email{Benjamin.Schraen@math.uvsq.fr}

\author{Benjamin Schraen}

\address{Laboratoire de Math\'ematiques de Versailles\\
UMR CNRS 8100\\
45 avenue des \'Etats Unis - B\^atiment Fermat\\
F--78035 Versailles Cedex, France}

\title[Représentations supersingulières]{Sur la pr\'esentation des repr\'esentations supersinguli\`eres de $\mathrm{GL}_2(\b)$}
\alttitle{On presentation of supersingular representations of $\mathrm{GL}_2(\b)$}

\begin{abstract}
Soit $\b$ une extension quadratique de $\mathbb{Q}_p$. Nous prouvons que les représentations lisses irréductibles supersingulières à caractère central de $\mathrm{GL}_2(\b)$ ne sont pas de présentation finie.
\end{abstract}

\begin{altabstract}
Let $\b$ be a quadratic extension of $\mathbb{Q}_p$. We prove that smooth irreducible supersingular representations with central character of $\mathrm{GL}_2(\b)$ are not of finite presentation.
\end{altabstract}

\subjclass{22E50,11F33}

\begin{document}

\maketitle

Le programme de Langlands $p$-adique permet de comprendre les liens existant entre représentations $p$-adiques continues du groupe $\mathrm{GL}_2(\mathbb{Q}_p)$ et les représentations $p$-adiques continues du groupe $\mathrm{Gal}(\overline{\mathbb{Q}_p}/\mathbb{Q}_p)$. Une première étape de ce programme, réalisée par C.~Breuil dans \cite{Brmodp}, a été de construire une correspondance $p$-modulaire semi-simple. Par la suite, P.~Colmez a montré dans \cite{Colmezfoncteur} que cette correspondance provient de la construction d'un foncteur de la catégorie des représentations lisses de $\mathrm{GL}_2(\mathbb{Q}_p)$ de longueur finie sur $\overline{\mathbb{F}_p}$ vers la catégorie des représentations continues de $\mathrm{Gal}(\overline{\mathbb{Q}_p}/\mathbb{Q}_p)$ sur $\overline{\mathbb{F}_p}$. Cette construction s'appuie de façon cruciale sur le fait que les représentations de $\mathrm{GL}_2(\mathbb{Q}_p)$ que l'on considère sont de présentation finie, et même mieux, peuvent être réalisées par les groupes d'homologie de systèmes de coefficients $\mathrm{GL}_2(\mathbb{Q}_p)$-équivariants sur l'arbre de Bruhat-Tits de $\mathrm{PGL}_{2,\mathbb{Q}_p}$. Ce résultat est une conséquence des travaux de L.~Barthel et R.~Livné (\cite{BLmodular}), C.~Breuil (\cite{Brmodp}), V.~Paskunas (\cite{Paskcoeff}), R.~Ollivier (\cite{OllivierIinv}), P.~Colmez (\cite{Colmezfoncteur}), Y.~Hu (\cite{HuDiag}) et M.~Emerton (\cite{EmCoh}).

Si $\b$ est une extension finie de $\mathbb{Q}_p$, M.-F.~Vignéras a construit dans \cite{Vigfoncteur} (voir aussi \cite{VigSch}) un foncteur associant à une représentation $p$-modulaire lisse admissible et de présentation finie, une représentation continue de dimension finie sur  $\overline{\mathbb{F}_p}$ du groupe $\mathrm{Gal}(\overline{\mathbb{Q}_p}/\mathbb{Q}_p)$. Cependant, nous ne connaissons pas pour l'instant de classification complète des représentations lisses irréductibles de $\mathrm{GL}_2(\b)$. C'est une conséquence des travaux de Barthel et Livné (\cite{BLmodular}) que les induites paraboliques sont de présentation finie. Le problème des représentations supersingulières, c'est-à-dire des représentations irréductibles qui ne sont pas des sous-quotients d'induites paraboliques, reste ouvert, bien que l'on sache qu'elles existent en grand nombre (\cite{Paskcoeff}, \cite{BrPask}).

Le but de ce travail est de répondre à la question de la présentation des représentations supersingulières de $\mathrm{GL}_2(\b)$ lorsque $\b$ est une extension quadratique de $\mathbb{Q}_p$. Notre résultat est le suivant.

\begin{theo}[\ref{theo:nonpf}]\label{theo:principal}
Supposons que $[\b:\mathbb{Q}_p]=2$. Une représentation lisse irréductible supersingulière de $\mathrm{GL}_2(\b)$ sur $\overline{\mathbb{F}_p}$ ayant un caractère central n'est pas de présentation finie.
\end{theo}

La preuve de ce résultat suit une stratégie inaugurée par Emerton dans \cite{EmCoh}. Considérons l'algèbre d'Iwasawa $\overline{\mathbb{F}_p}[[U]]$ où $U\subset\mathrm{GL}_2(\mathcal{O}_\b)$ est le sous-groupe des matrices triangulaires supérieures à coefficients entiers. La matrice $\alpha=\left(\begin{smallmatrix}\varpi&0\\0&1\end{smallmatrix}\right)$ induit par conjugaison un endomorphisme plat $\phi$ de l'algèbre $\overline{\mathbb{F}_p}[[U]]$ et donne ainsi une structure $\phi$-module à coefficients dans $\overline{\mathbb{F}_p}[[U]]$ à toute représentation lisse de $\mathrm{GL}_2(\b)$. Emerton a remarqué que l'étude de certains sous-$\phi$-modules d'une telle représentation est facilitée par le fait que l'anneau $\overline{\mathbb{F}_p}[[U]][X]_{\phi}$ des polynômes tordus par $\phi$ est cohérent à gauche.

Dans l'article \cite{EmCoh}, Emerton étudie les modules sur l'anneau $A[X]_\phi$, où $A$ est un anneau commutatif noethérien et $\phi$ un endomorphisme plat de $\phi$. Il montre que $A[X]_\phi$ est cohérent à gauche et étudie plus particulièrement le cas où $A$ est un anneau de valuation discrète. Nous nous intéressons ici au cas plus général où $A$ est un anneau local noethérien complet et régulier, ce qui est le cas de $\overline{\mathbb{F}_p}[[U]]$. Contrairement au cas d'un anneau de valuation discrète, il n'est plus vrai en général que tout $\overline{\mathbb{F}_p}[[U]][X]_{\phi}$-module $M$ de type fini pour lequel $\overline{\mathbb{F}_p}\otimes_{\overline{\mathbb{F}_p}[[U]]}M$ est de dimension finie sur $\overline{\mathbb{F}_p}$ soit de corang fini sur $\overline{\mathbb{F}_p}[[U]]$.

Pour prouver le théorème \ref{theo:principal}, nous utilisons un théorème de Hu (\cite{HuDiag}) affirmant que si une représentation $\pi$ vérifie les hypothèses du théorème précédent et est de présentation finie, alors nécessairement un sous-espace $I^+(\pi,\sigma)\subset\pi$ doit être de corang fini sur l'algèbre $\overline{\mathbb{F}_p}[[U]]$. Notre preuve est basée sur l'étude de $I^+(\pi,\sigma)$ en tant que $\overline{\mathbb{F}_p}[[U]][X]_{\phi}$-module. En utilisant les résultats de \cite{HMS}, nous montrons en particulier que si $I^+(\pi,\sigma)$ est un $\overline{\mathbb{F}_p}[[U]]$-module de corang fini, il est de présentation finie en tant que $\overline{\mathbb{F}_p}[[U]][X]_{\phi}$-module. Puis nous prouvons que si $I^+(\pi,\sigma)$ est de présentation finie, alors la représentation $\ind_{Z\mathrm{GL}_2(\mathcal{O}_\b)}^{\mathrm{GL}_2(\b)}(\sigma)/T$ est admissible, ce qui est faux dès que $\b\neq\mathbb{Q}_p$.\\

La première partie de cet article est consacrée à l'étude des $A[X]_\phi$-modules. Nous commençons par rappeler quelques résultats de l'article \cite{EmCoh}. Nous nous spécialisons très vite au cas où $A$ est un anneau noethérien local complet et régulier. Nous faisons quelques rappels sur la dualité de Pontryagin des $A$-modules localement artiniens et introduisons la caractéristique d'Euler-Poincaré d'un $A[X]_\phi$-module de présentation finie. Ensuite nous montrons le résultat clé sur les $A[X]_\phi$-modules qui sont de corang fini en tant que $A$-module.\\
Dans une deuxième partie nous appliquons les généralités sur les $A[X]_\phi$-modules aux représentations supersingulières de $\mathrm{GL}_2(\b)$. Nous commençons par rappeler un certain nombre de résultats sur les représentations de présentation finie, ainsi que la définition de $I^+(\pi,\sigma)$, puis nous détaillons certains calculs propres aux représentations supersingulières, en particulier la non admissibilité de $\ind_{Z\mathrm{GL}_2(\mathcal{O}_\b)}^{\mathrm{GL}_2(\b)}(\sigma)/T$. La dernière section de cette deuxième partie conclut la preuve du théorème \ref{theo:principal}.\\

\section{$\phi$-modules sur $\C[[U]]$}

\subsection{Généralités}

Soit $A$ un anneau commutatif noethérien et $\phi$ un endomorphisme plat de $A$. On note $A[X]_\phi$ l'anneau des polynômes tordus. Ses éléments sont les polynômes $\sum_{i=0}^n a_i X^i$ et il est muni de la loi de multiplication
\begin{equation*}
(\sum_{i=0}^n a_iX^i)(\sum_{j=0}^m b_jX^j)=\sum_{k=0}^{n+m}(\sum_{i+j=k}a_i\phi^i(b_j))X^k.
\end{equation*}
Nous abrégerons toujours $A[X]_\phi$-module à gauche par $A[X]_\phi$-module. Si $M$ est un $A[X]_\phi$-module, on note $\phi_M$ l'endomorphisme du groupe abélien $M$ défini par l'action de l'élément $X$.

\begin{defi}
Notons $\mathcal{A}_\phi$ le $A$-bimodule suivant. Son groupe abélien sous-jacent est $A$. Sa structure de $A$-module à gauche est induite par la multiplication et sa structure à droite est celle pour laquelle $a \in A$ agit par multiplication par $\phi(a)$.

Si $M$ est un $A$-module, on note $\phi^*M$ le $A$-module $\mathcal{A}_\phi\otimes_A M$. Pour $n \geq 0$, on note $(\phi^*)^n$ la composée $n$ fois du foncteur $\phi^*$, $(\phi^*)^0$ désignant le foncteur identité.
\end{defi}
Comme $\phi$ est un endomorphisme plat de $A$, tous les foncteurs $(\phi^*)^n$ sont exacts.

\begin{lemm}\label{lemm:exact}
Le foncteur $A[X]_\phi\otimes_A -$ de la catégorie des $A$-modules vers la catégorie des $A[X]_\phi$-modules est exact.
\end{lemm}

\begin{proof}
Si $M$ est un $A$-module, on a un isomorphisme $A[X]_\phi \otimes_AM\xrightarrow{\sim}\bigoplus_{n \geq 0} (\phi^*)^n(M)$ défini par $\sum_{i=0}^na_iX^i\otimes b\mapsto (a_i\otimes1\otimes\cdots\otimes1\otimes b)_{i\geq0}$. Comme $\phi$ est un endomorphisme plat de $A$, chaque foncteur $\phi^{*n}$ est exact, de même que le foncteur somme directe.
\end{proof}

D'après \cite[Prop. $1.3$]{EmCoh}, l'anneau $A[X]_\phi$ est cohérent à gauche, autrement dit, un sous-$A[X]_\phi$-module de type fini d'un $A[X]_\phi$-module de présentation finie est de présentation finie. Si $M$ est un $A$-module de type fini, le $A[X]_\phi$-module $A[X]_\phi\otimes_A M$ est un exemple de $A[X]_\phi$-module de présentation finie.

\subsection{Produits de torsion}

Soit $f : \, A \rightarrow B$ un morphisme d'anneaux noethériens. Supposons que $B$ soit muni d'un endomorphisme plat $\phi'$ et que $\phi' \circ f=f \circ \phi$. D'après \cite[Lemma $2.1$]{EmCoh}, il existe des isomorphismes de $\delta$-foncteurs $\mathrm{Tor}_i^A(B, -) \simeq \mathrm{Tor}_i^{A[X]_\phi}(B[X]_{\phi'}, -)$, munissant chaque groupe $\mathrm{Tor}_i^A(B,M)$ d'une structure de $B[X]_{\phi'}$-module. De plus, si $M$ est un $A[X]_\phi$-module de présentation finie, chaque $\tor_i^A(B,M)$ est un $B[X]_{\phi'}$-module de présentation finie (\cite[Prop. $2.2$]{EmCoh}).

\begin{prop}\label{prop:tenseur}
Il existe des isomorphismes entre $\delta$-foncteurs de la catégorie des $A$-modules vers la catégorie des $B[X]_{\phi'}$-modules $\tor_i^A(B,A[X]_\phi\otimes_A-) \simeq B[X]_{\phi'}\otimes_B\tor_i^A(B,-)$ pour tout $i \geq 0$.
\end{prop}

\begin{proof}
Soit $G$ le foncteur $B\otimes_A(A[X]_\phi\otimes_A-)$. Par associativité du produit tensoriel, il existe des isomorphismes canoniques de foncteurs
\begin{equation*}
G\simeq B[X]_{\phi'}\otimes_A- \ \mathrm{et} \ G\simeq B[X]_{\phi'}\otimes_B (B \otimes_A-). 
\end{equation*}
D'après le lemme \ref{lemm:exact}, le foncteur $A[X]_\phi\otimes_A -$ est exact et transforme un $A$-module libre en un $A$-module plat. On a donc des isomorphismes de $\delta$-foncteurs $L_i G\simeq \tor_i^A(B,A[X]_\phi\otimes_A -)$. De même, le foncteur $B\otimes_A -$ transforme $A$-modules libres en $B$-modules libres, et le foncteur $B[X]_{\phi'}\otimes_B -$ est exact, d'où un isomorphisme de $\delta$-foncteurs $L_i G \simeq B[X]_{\phi'}\otimes_B \tor_i^A(B,-)$.
\end{proof}

\subsection{Cas d'un anneau local régulier complet}\label{sec:regulier}

Nous supposons désormais que l'anneau $A$ est local régulier complet de dimension finie $d$, de corps résiduel $\C$ et d'idéal maximal $\mathfrak{m}$. Supposons de plus l'endomorphisme $\phi$  local induisant l'identité sur $\C$. En particulier, si $M$ est un $A[X]_\phi$-module, les $\C$-espaces vectoriels $\tor_i^A(\C,M)$ sont munis d'une structure de $\C[X]$-espace vectoriel, $\C[X]$ désignant désormais l'anneau commutatif des polynômes en $X$. De plus si $N$ désigne un $\C[X]$-module, nous notons $N_{tors}$ le sous $\C[X]$-module de $\C[X]$-torsion de $N$.
\subsubsection{}Dans la suite, nous nous intéresserons plus particulièrement aux $A$-modules localement artiniens.
\begin{defi}
On dit qu'un $A[X]_\phi$-module est lisse si le $A$-module sous-jacent est localement artinien.
\end{defi}

\begin{exem}\label{exem:calcul}
Soit $M$ un $A$-module localement artinien. Le $A[X]_\phi$-module $A[X]_\phi\otimes_A M$ est lisse et la proposition \ref{prop:tenseur} montre que pour tout $i$, le $\C[X]$-module $\tor_i^A(\C,A[X]_\phi\otimes_A M)$ est un $\C[X]$-module libre de rang $\dim_\C\tor_i^A(\C,M)$. Par exemple si $I$ est un idéal ouvert de $A$ engendré par $d$ éléments, chaque $\tor_i^A(\C,A[X]_\phi\otimes_A(A/I))$ est un $\C[X]$-module libre de rang $\binom{d}{i}$.
\end{exem}

Si $f : \, M \rightarrow N$ est un morphisme de $A$-modules, on note dans la suite $\tor_i^A(f)$ l'application $\C$-linéaire induite par fonctorialité $\tor_i^A(\C,M) \rightarrow \tor_i^A(\C,N)$.

\begin{exem}
Soit $M$ un $A$-module. L'isomorphisme $A[X]_\phi\otimes_AM\simeq\bigoplus_{n\geq0}(\phi^*)^nM$ permet d'identifier chaque $(\phi^*)^nM$ à un facteur direct du $A$-module $A[X]_\phi\otimes_AM$, et donc chaque $\tor_i^A(\C,(\phi^*)^nM)$ à un sous-$\C$-espace vectoriel de $\tor_i^A(\C,A[X]_\phi\otimes_AM)$. La proposition \ref{prop:tenseur} montre alors que la multiplication par $X^n$ sur $\tor_i^A(\C,A[X]_\phi\otimes_AM)$ induit un isomorphisme de $\C$-espaces vectoriels $\tor_i^A(\C,M)\simeq\tor_i^A(\C,(\phi^*)^nM)$. On obtient ainsi un isomorphisme de foncteurs en $\C$-espaces vectoriels $\tor_i^A(\C,-)\simeq\tor_i^A(\C,(\phi^*)^n-)$. Plus précisément si $f:\,M\rightarrow N$ est un morphisme de $A$-modules, on a un diagramme commutatif
\begin{equation}\label{eq:diagTor}
\xymatrix{\tor_i^A(\C,M)\ar[d]^{\tor_i^A(f)}\ar[r]^{X^n}_{\sim}&\tor_i^A(\C,(\phi^*)^nM)\ar[d]^{\tor_i^A((\phi^*)^nf)}\\ \tor_i^A(\C,N)\ar[r]^{X^n}_{\sim}&\tor_i^A(\C,(\phi^*)^nN).}
\end{equation}
\end{exem}

\subsubsection{}Comme $A$ est un anneau local régulier de dimension finie $d$, on a $\tor_i^A(-,-)=0$ pour $i > d$. En particulier le foncteur $\tor_d^A(\C,-)$ est exact à gauche et est, non canoniquement, isomorphe au foncteur $M \mapsto M[\mathfrak{m}]$ des éléments annulés par $\mathfrak{m}$. Notons de plus que si $M$ est un $A$-module localement artinien, alors $\tor_d^A(\C,M)\neq 0$ si et seulement si $M \neq 0$.

\begin{lemm}\label{lemm:inj}
Soient $h : \, M \rightarrow N$ un homomorphisme de $A$-modules localement artiniens. Si l'application $\tor_d(h)$ est injective, alors $h$ est injective.
\end{lemm}

\begin{proof}
Soit $L$ le noyau de $h$. Le foncteur $\tor_d^A(\C, -)$ est exact à gauche, on a donc $\tor_d^A(\C, L)=0$. Or $L$ est localement artinien, donc $L=0$ si et seulement si $\tor_d^A(\C,L)=0$.
\end{proof}

\subsubsection{}Soit $I$ un module duaisant pour $A$ au sens de \cite[Déf. IV.$4.1$]{SGA2}. Un tel module existe toujours d'après \cite[Thm. IV.$4.7$]{SGA2}. Si $M$ est un $A$-module de longueur finie, on note $M^{\vee}=\hom_A(M,I)$. C'est un $A$-module de longueur finie et on a un isomorphisme de foncteurs de la catégorie des $A$-modules de longueur finie vers elle-même $\mathrm{Id}\simeq ((-)^{\vee})^{\vee}$. Soit $\promod$ la catégorie des $A$-modules pseudocompacts (\cite[VII.B.$0.2$]{SGA31}) et $\locart$ celle des $A$-modules localement artiniens. Si $M$ est un $A$-module localement artinien, $M^{\vee}=\hom_A(M,I)$, muni de la topologie de la convergence simple est un $A$-module pseudocompact. Réciproquement si $M$ est un $A$-module pseudocompact, le $A$-module $M^{\vee}=\hom_A^{cont}(M,I)$, $I$ étant muni de la topologie discrète, est un $A$-module localement artinien. On obtient ainsi deux foncteurs contravariants $(-)^{\vee}:\,\locart\rightarrow\promod$ et $(-)^{\vee}:\,\promod\rightarrow\locart$ et des isomorphismes de foncteurs $\mathrm{Id}_{\locart}\simeq ((-)^{\vee})^{\vee}$ et $\mathrm{Id}_{\promod}\simeq ((-)^{\vee})^{\vee}$.

\'Equivalente à $(\locart)^{op}$, la catégorie $\promod$ est abélienne et possède suffisamment de projectifs. Comme $A$ est noethérien, l'idéal $\mathfrak{m}$ et de type fini, et donc le foncteur $k\otimes_A(-)$ va de la catégorie des $A$-modules pseudocompacts vers la catégorie des $\C$-modules pseudocompacts. \'Etant exact à droite, on peut donc définir ses foncteurs dérivés à gauche. Comme $\C$ a une résolution par des $A$-modules libres de type finis, en particulier projectifs dans la catégorie $\promod$, ces foncteurs coïncident avec les foncteurs $\tor_i^A(k,-)$. Autrement dit, si $M$ est un $A$-module pseudocompact, $\tor_i^A(k,M)$ est naturellement muni d'une structure de $\C$-module pseudocompact.

\begin{prop}\label{prop:dual}
Pour tout $0 \leq i \leq d$, il existe des isomorphismes de foncteurs de la catégorie des $A$-modules lisses vers la catégorie des $\C$-modules profinis
\begin{equation*}
\tor_i^A(\C,-)^{\vee} \simeq \tor_{d-i}^A(\C,(-)^{\vee}).
\end{equation*}
\end{prop}

\begin{proof}
Rappelons que si $C_\bullet$ désigne un complexe descendant de $A$-modules, on note $C^\bullet$ le complexe ascendant défini par $C^i=C_{-i}$. De plus, si $C_\bullet$ est un complexe descendant de $A$-modules localement artiniens, le complexe $C_\bullet^\vee$ est un complexe ascendant de $A$-modules pseudocompacts. Soit $x=(x_1,\dots,x_d)$ une suite régulière engendrant $\mathfrak{m}$ et $K_{\bullet}(x)$ le complexe de Koszul associé. Ils s'agit d'une résolution libre de type fini de $\C$. Si $M$ est un $A$-module localement artinien, le complexe $K_{\bullet}(x,M)=K_{\bullet}(x)\otimes_AM$ est un complexe de $A$-modules localement artiniens. Comme le foncteur $(-)^{\vee}$ est exact, on a un isomorphisme de foncteurs $H^i(K_\bullet(x,M)^{\vee})\simeq \tor_i^A(\C,-)^{\vee}$. Comme chaque $K_j(x)$ est un $A$-module libre de type fini, on a un isomorphisme de complexes de foncteurs $K_{\bullet}(x,-)^{\vee}\simeq\hom_A(K_{\bullet}(x),A)\otimes_A(-)^{\vee}$. Ainsi on a un isomorphisme de foncteurs $H^{i}(K_{\bullet}(x,-)^{\vee}) \simeq H^{i}(K_{\bullet}(x)\otimes_A (-)^{\vee})$. On utilise alors l'isomorphisme de complexes de Koszul $\hom_A(K_{\bullet}(x),A)\simeq K^{\bullet}(x)[d]$ (dépendant du choix d'un isomorphisme de $A$-modules $\Omega^d_A\simeq A$) pour conclure que $H^{i}(K_{\bullet}(x,-)^{\vee})\simeq H_{d-i}(K_{\bullet}(x)\otimes_A(-)^{\vee})\simeq\tor_{d-i}(\C,(-)^{\vee})$.
\end{proof}

\begin{exem}
Soient $p$ un nombre premier et $U$ un $\mathbb{Z}_p$-module libre de type fini de rang $d$. C'est naturellement un pro-$p$-groupe uniforme. Supposons $\C$ de caractéristique $p$ et notons $A=\C[[U]]$ l'algèbre d'Iwasawa de $U$. Alors $A$ est un anneau local noethérien régulier de dimension $d$ et de corps résiduel $\C$. Dans ce cas, il existe un isomorphisme $\tor_d^A(\C, -) \simeq ( -)^{U}$. Dans cet exemple, un $\C[[U]]$-module localement artinien est une $\C$-représentation lisse de $U$. Si $V$ est une $\C$-représentation lisse de $U$, on note $T(V)$ le $\C$-espace vectoriel $\hom_\C(U,\C)$ que l'on munit de la topologie de la convergence simple. Le groupe $U$ agit sur cet espace par $g\cdot f(\cdot)=f(g^{-1}\cdot)$. Cette action est continue pour la topologie de $T(V)$, et la structure de $\C[U]$-module sur $T(V)$ se prolonge de façon unique en une structure de $\C[[U]]$-module pseudocompact. L'unicité du foncteur dualisant (\cite[Thm. IV.$4.7$]{SGA2}) montre qu'il existe un isomorphisme de foncteur $T(-)\simeq (-)^{\vee}$. Dans ce cas le dual d'un $\C[[U]]$-module localement artinien n'est rien d'autre que la construction usuelle de la représentation duale.
\end{exem}

\begin{defi}
Un $A$-module $M$ est dit admissible, ou de corang fini, s'il est localement artinien et si $\mathrm{Tor}_d^A(\C,M)$ est de dimension finie sur $\C$.
\end{defi}
La proposition \ref{prop:dual} et le lemme de Nakayama (\cite[VII.B.$0.3.4$]{SGA31}) impliquent que $M$ est admissible si et seulement si le $A$-module $M^{\vee}$ est un $A$-module de type fini. Comme $A$ est un anneau noethérien, tout sous-module ou quotient d'un $A[X]_\phi$-module admissible est admissible. On a de plus le résultat suivant.

\begin{coro}\label{coro:adm}
Si $M$ est un $A$-module admissible, alors tous les $\C$-espaces vectoriels $\tor_i^A(\C,M)$ sont de dimension finie.
\end{coro}

\begin{proof}
La $A$-module $M^{\vee}$ est de type fini, donc comme $A$ est noethérien, tous les $\tor_i^A(\C,M^{\vee})$ sont des $\C$-modules de type fini. La proposition \ref{prop:dual} montre alors qu'il en est de même des $\tor_i^A(\C,M)$ puisque $\C$ est un corps.
\end{proof}

\subsubsection{}Le lemme suivant nous servira plusieurs fois dans la suite.

\begin{lemm}\label{lemm:suite}
Soit $M$ un $A[X]_\phi$-module lisse de présentation finie. Il existe une suite croissante $(M_i)_{i \geq 0}$ de sous-$A[X]_\phi$-modules de $M$ et une suite $(V_i)_{i\geq 0}$ de $\C$-espaces vectoriels de dimension finie telles que
\begin{itemize}
\item
pour tout $i\geq 0$, il existe un isomorphisme de $A[X]_\phi$-modules $M_{i+1}/M_i \simeq A[X]_\phi\otimes_A V_i$ ;
\item
si $\widetilde{M}=\bigcup_{i \geq 0} M_i$, alors l'application quotient $M\rightarrow M/\widetilde{M}$ induit un isomorphisme $\tor_d^A(\C,M)_{tors}\simeq\tor_d^A(\C,M/ \widetilde{M})$.
\end{itemize}
En particulier, le $A[X]_\phi$-module $M/ \widetilde{M}$ est admissible et chaque $M_i$ est un $A[X]_\phi$-module de présentation finie.
\end{lemm}

\begin{proof}
Rappelons qu'il existe un isomorphisme de foncteurs $\tor_d^A(\C,-)\simeq (-)[\mathfrak{m}]$ que nous fixons. Nous allons construire les deux suites par récurrence de telle sorte que pour tout $i$, l'application $r_i:\,M/M_i\rightarrow M/M_{i+1}$, obtenue par passage au quotient, induise un isomorphisme $\tor_d^A(r_i):\,\tor_d^A(\C,M/M_i)_{tors} \simeq \tor_d^A(\C,M/M_{i+1})_{tors}$ et que l'image de $\tor_d^A(r_i)$ soit contenue dans la torsion de $\tor_d^A(\C,M/M_{i+1})$. On pose $M_0=0$ et $V_0=0$. Supposons que nous ayions construit $M_i \subset M$. Comme $M/M_i$ est un $A[X]_\phi$-module de présentation finie, le $\C[X]$-module $\tor_d^A(\C,M/M_i)$ est un $\C[X]$-module de type fini que l'on peut donc écrire sous la forme $\tor_d^A(\C,M/M_i)_{tors} \oplus \C[X]^{n_i}$. Soit $V_i \subset \tor_d^A(\C,M/M_i) \subset M/M_i$ un sous-$\C$-espace vectoriel de dimension $n_i$ engendrant le sous-$\C[X]$-module $\C[X]^{n_i}$. Considérons $V_i$ comme un $A$-module via le morphisme résiduel $A \twoheadrightarrow \C$ et posons $\overline{M_{i+1}}=A[X]_\phi\otimes_A V_i$. L'inclusion $V_i \subset M/M_i$ induit un morphisme de $A[X]_\phi$-modules $\overline{M_{i+1}} \rightarrow M/M_i$. Ce morphisme est injectif. En effet, l'exemple \ref{exem:calcul} montre que $\tor_d^A(\C,\overline{M_{i+1}}) \simeq \C[X]\otimes_\C V_i$ et l'image de $V_i \subset \overline{M_{i+1}}$ dans $\tor_d^A(\C,M/M_i)$ engendre alors un $\C[X]$-module libre de rang $\dim_\C V_i$, ainsi l'application $\tor_d^A(\C,\overline{M_{i+1}}) \rightarrow \tor_d^A(\C,M/M_i)$ est injective. D'après le lemme \ref{lemm:inj}, l'application $\overline{M_{i+1}} \rightarrow M/M_i$ est injective. On note alors $M_{i+1} \subset M$ l'image réciproque de $\overline{M}_{i+1}$ par $M \rightarrow M/M_i$. Le sous-$A[X]_\phi$-module $M_{i+1}\subset M$ est clairement un sous-$A[X]_\phi$-module de présentation finie. De plus on a une suite exacte
\begin{equation*}
0 \rightarrow \tor_d^A(\C,M/M_i)_{tors} \rightarrow \tor_d^A(\C,M/M_{i+1}) \rightarrow \tor_{d-1}^A(\C,\overline{M_{i+1}}) \rightarrow \tor_{d-1}^A(\C,M/M_i).
\end{equation*}
D'après l'exemple \ref{exem:calcul}, le $\C[X]$-module $\tor_{d-1}^A(\C,\overline{M_{i+1}})$ est libre de rang $d n_i$, ainsi tous ses sous-modules sont des $\C[X]$-modules libres. On a donc un isomorphisme de $\C[X]$-modules $\tor_d^A(\C, M/M_{i+1}) \simeq \tor_d^A(\C,M/M_i)_{tors} \oplus \C[X]^{n_{i+1}}$. En particulier l'application $r_i : \, M/M_i \rightarrow M/M_{i+1}$ induit un morphisme de $\C[X]$-modules $\tor_d^A(\C,M/M_i) \rightarrow \tor_d^A(\C,M/M_{i+1})$ dont l'image est de torsion et sa restriction à $\tor_d^A(\C,M/M_i)_{tors}$ est un isomorphisme $\tor_d^A(\C,M/M_i)_{tors} \simeq \tor_d^A(\C,M/M_{i+1})_{tors}$. Ceci achève la récurrence.\\
Pour conclure, remarquons que si $\widetilde{M}=\bigcup_{i \geq 0} M_i$, on a $M/ \widetilde{M}=\varinjlim_i (M/M_i)$ et donc $\tor_d^A(\C,M/ \widetilde{M})=\varinjlim_i \tor_d^A(\C,M/M_i)=\tor_d^A(\C,M)_{tors}$ par construction.
\end{proof}

\subsection{Caractéristique d'Euler-Poincaré}

Si $M$ est un $A[X]_\phi$-module de présentation finie, le $\C[X]$-module $\tor_i^A(\C,M)$ est de type fini, on note $h_i(M)=\dim_{\C(X)}(\C(X)\otimes_{\C[X]}\tor_i^A(\C,M))$ le rang de sa partie libre. On pose alors
\begin{equation*}
\chi^{A,\phi}(M)=\sum_{i \geq 0} (-1)^ih_i(M).
\end{equation*}
Comme $\C(X)$ est plat sur $\C[X]$, on voit immédiatement que si $0\rightarrow M'\rightarrow M\rightarrow M''\rightarrow 0$ est une suite exacte courte de $A[X]_\phi$-modules de présentation finie, on a $\chi^{A,\phi}(M)=\chi^{A,\phi}(M')+\chi^{A,\phi}(M'')$.\\
En particulier un $A[X]_\phi$-module de présentation finie est admissible si et seulement si $h_d(M)=0$, et dans ce cas, le corollaire \ref{coro:adm} montre que $\dim_\C\tor_i^A(\C,M)$ est fini pour tout $i$, c'est-à-dire $h_i(M)=0$ pour tout $i$. Dans \cite[Prop. $3.5$]{EmCoh}, Emerton montre que si $d=1$, un $A[X]_\phi$-module lisse $M$ de type fini est admissible si et seulement si $h_0(M)=0$. Le résultat qui va suivre est une généralisation partielle de ce résultat.\\

Nous supposons désormais dans cette section que $\C$ est un corps de caractéristique $p$ et $A$ une $\C$-algèbre noethérienne locale complète lisse sur $\C$ de dimension $d$ et de corps résiduel $\C$. Il existe donc un isomorphisme de $\C$-algèbres $A\simeq\C[[X_1,\dots, X_d]]$. Choisissons un tel isomorphisme, on peut alors définir un $\C$-endomorphisme plat de $A$ par $\phi(X_i)=X_i^q$ où $q$ est une puissance de $p$.

\begin{prop}\label{prop:EP}
Si $M$ est un $A[X]_\phi$-module lisse de présentation finie, on a $\chi^{A,\phi}(M)=0$.
\end{prop}

\begin{proof}
Commençons par traiter le cas où $d=1$, autrement dit où $A$ est un anneau de valuation discrète. Soient $(M_i)$ et $(V_i)$ des suites comme dans le lemme \ref{lemm:suite}. Posons $\widetilde{M}=\bigcup_i M_i$. Comme chaque quotient $M_{i+1}/M_i$ est de la forme $A[X]_\phi\otimes_A V_i$, l'exemple \ref{exem:calcul} et l'additivité de la fonction $\chi^{A,\phi}$ montrent que $\chi^{A,\phi}(M_i)=0$ pour tout $i$. De plus $M/\widetilde{M}$ est un $A[X]_\phi$-module de type fini et admissible, donc $\chi^{A,\phi}(M/\widetilde{M})=0$. De plus $M/\widetilde{M}$ est de présentation finie par \cite[Prop. $3.2$]{EmCoh}, donc la cohérence de $A[X]_\phi$ montre que $\widetilde{M}$ est de présentation finie, en particulier de type fini. Il existe donc $i$ tel que $M_i=\widetilde{M}$. Comme on a déjà vu que $\chi^{A,\phi}(M_i)=0$, l'additivité de $\chi^{A,\phi}$ implique que $\chi^{A,\phi}(M)=0$.\\
On procède alors par récurrence sur la dimension de $A$.\\
Posons $B=A/(X_1)$. L'idéal $(X_1)$ étant stable par $\phi$, l'application $A \rightarrow B$ est $\phi$-équivariante. On a une suite spectrale de $\C[X]$-modules
\begin{equation*}
E_{i,j}^2=\tor_i^B(\C, \tor_j^A(B, M)) \Rightarrow \tor_{i+j}^A(\C,M).
\end{equation*}
Une résolution libre du $A$-module $B$ est donnée par $A \xrightarrow{X_1} A$, donc $\tor_j^A(B,M)=0$ pour $j \geq 2$. On en déduit une suite exacte longue de $\C[X]$-modules
\begin{equation*}
\cdots \rightarrow \tor_{i-1}^B(\C, \tor_1^A(B,M)) \rightarrow \tor_i^A(\C,M) \rightarrow \tor_i^B(\C,B\otimes_A M) \rightarrow \cdots
\end{equation*}
D'où une égalité $\chi^{A,\phi}(M)=\chi^{B,\phi}(B \otimes_A M)-\chi^{B,\phi}(\tor_1^A(B,M))$. Or les deux $B[X]_\phi$-modules $B \otimes_A M$ et $\tor_1^A(B,M)$ sont de présentation finie et lisses. Par récurrence sur la dimension de $A$, on a $\chi^{A,\phi}(M)=0$.
\end{proof}

\begin{rema}
L'auteur ignore si ce résultat reste vrai en supposant uniquement $A$ noethérien local complet et régulier, et pour tout endomorphisme plat.
\end{rema}

\subsection{Admissibilité de certains $A[X]_\phi$-modules}

Si $A$ est un anneau de valuation discrète, un $A[X]_\phi$-module admissible de type fini et admissible est de présentation finie (\cite[Prop. $3.2$]{EmCoh}).  En général ce n'est plus vrai si $A$ est un anneau local régulier de dimension plus grande que $1$.

On conserve les hypothèses de la section \ref{sec:regulier}, mais en supposant que $d=2$.

\begin{prop}\label{prop:admnonpf}
Soit $M$ un $A[X]_\phi$-module de présentation finie tel que $\chi^{A,\phi}(M)=0$ et $h_0(M)=0$. Alors soit $M$ est admissible, soit il existe un sous-$A[X]_\phi$-module $\widetilde{M} \subset M$ qui n'est pas de type fini et tel que $\tor_2^A(\C,M/ \widetilde{M})\simeq\tor_2^A(\C,M)_{tors}$. En particulier $M/ \widetilde{M}$ est admissible, de type fini, mais n'est pas de présentation finie.
\end{prop}

\begin{proof}
Remarquons que les conditions $h_0(M)=0$ et $\chi^{A,\phi}(M)=0$ impliquent $h_2(M)=h_1(M)$, et que cette valeur commune est non nulle si et seulement si $M$ n'est pas admissible. Soient $(M_i)_{i \geq 0}$ et $(V_i)_{i \geq 0}$ des suites vérifiant les hypothèses du lemme \ref{lemm:suite}. Supposons $M$ non admissible et montrons alors que la suite $(M_i)_{i \geq 0}$ est strictement croissante. Il suffit pour cela prouver que $\forall i \geq 0$, $n_i:=\dim_\C(V_i) \geq 1$. Par construction $n_0=h_2(M) \geq 1$. Supposons alors que pour un $i \geq 0$, on ait $n_i \geq 1$. On a une suite exacte
\begin{equation*}
0 \rightarrow \tor_2^A(\C,M/M_{i})_{tors} \rightarrow \tor_2^A(\C,M/M_{i+1}) \rightarrow \tor_1^A(\C,\overline{M_{i+1}}) \rightarrow \tor_1^A(\C,M/M_{i}).
\end{equation*}
D'après l'exemple \ref{exem:calcul}, le $\C[X]$-module $\tor_1^A(\C,\overline{M_{i+1}})$ est libre de rang $2n_i$. Comme $M/M_{i}$ est un quotient de $M$, on a $h_0(M/M_{i}) \leq h_0(M)=0$, donc $h_0(M/M_{i})=0$. On a $\chi^{A,\phi}(M_{i})=\chi^{A,\phi}(M)=0$, donc $\chi^{A,\phi}(M/M_{i})=0$ et $\tor_1^A(\C,M/M_{i})$ est un $\C[X]$-module de rang $n_i$. On en déduit $n_{i+1} \geq n_i \geq 1$, ce qui achève la récurrence.\\
Posons donc $\widetilde{M}=\bigcup_{i \geq 0} M_i$. On a $M/\widetilde{M}=\varinjlim_i (M/M_i)$, donc $\tor_2^A(\C,M/\widetilde{M})=\varinjlim_i \tor_2^A(\C,M/M_i)=\tor_2^A(\C,M)_{tors}$. Comme $M$ est de présentation finie, le module $\tor_2^A(\C,M)_{tors}$ est fini, donc $M/\widetilde{M}$ est admissible. Mais $M/\widetilde{M}$ n'est pas de présentation finie car s'il l'était, par cohérence de $A[X]_\phi$, $\widetilde{M}$ serait de type fini, et donc $\widetilde{M}=M_i$ pour un certain $i$, contredisant le caractère strictement croissant de la suite $(M_i)$.
\end{proof}

\begin{rema}
Sous les hypothèses de la proposition \ref{prop:EP}, l'hypothèse $\chi^{A,\phi}(M)=0$ est automatiquement vérifiée.
\end{rema}

\section{Applications aux représentations de $\mathrm{GL}_2(\b)$}

Soit $p$ un nombre premier. On fixe $\b$ une extension finie de degré $d$ de $\mathbb{Q}_p$. On note $\mathcal{O}_\b$ son anneau d'entiers, $\res$ son corps résiduel, $f$ le degré de $\res$ sur $\mathbb{F}_p$, $e$ l'indice de ramification de $\b$ sur $\mathbb{Q}_p$ et on choisit $\varpi\in\mathcal{O}_\b$ une uniformisante de $\b$. Si $e=1$, on fait le choix naturel $\varpi=p$.

On pose $G=\mathrm{GL}_2(\b)$, on désigne par $Z$ son centre. On note $\K=Z\mathrm{GL}_2(\mathcal{O}_\b)\subset G$, $I\subset\mathrm{GL}_2(\mathcal{O}_\b)$ le sous-groupe d'Iwahori supérieur et $I_1$ le pro-$p$-sous-groupe de Sylow de $I$. On note $K_1=\mathrm{Id}+\varpi\mathrm{M}_2(\mathcal{O}_\b)\subset\mathrm{GL}_2(\mathcal{O}_\b)$ et $\Gamma=\mathrm{GL}_2(\mathcal{O}_\b)/K_1\simeq\mathrm{GL}_2(\res )$. Le groupe $G$ contient les éléments particuliers suivants : $w=\left( \begin{smallmatrix} 0 & 1 \\ 1 & 0 \end{smallmatrix}\right)$, $\alpha=\left(\begin{smallmatrix}\varpi&0\\0&1\end{smallmatrix}\right)$, $\beta=w \alpha$. On désigne par $U$ le sous-groupe des matrices unipotentes supérieures de $\mathrm{GL}_2(\mathcal{O}_\b)$. De même, on pose $U^-$ le sous-groupe des matrices unipotentes inférieures de $\mathrm{GL}_2(\mathcal{O}_\b)$. Enfin on note $H$ le sous-groupe de $\mathrm{GL}_2(\mathcal{O}_\b)$ constitué des matrices $\left( \begin{smallmatrix} a & 0 \\ 0 & 1 \end{smallmatrix}\right)$ avec $a\in\mathcal{O}_\b^{\times}$ et $\widetilde{H}=HU$ le sous-groupe de $G$ engendré par $H$ et $U$.

On fixe désormais $\C$ une extension de $\mathbb{F}_p$ et un plongement $\tau_0:\,k_\b\hookrightarrow k$. Toutes les représentations apparaissant dans la suite sont des représentations $\C$-linéaires.

\subsection{Représentations de présentation finie}

Nous posons désormais $A=\C[[U]]$ l'algèbre d'Iwasawa de $U$. Comme $U$ est un $\mathbb{Z}_p$-module libre de rang $d$, $A$ est bien une $\C$-algèbre locale noethérienne régulière de corps résiduel isomorphe à $\C$. On a $\alpha\widetilde{H}\alpha^{-1}\subset\widetilde{H}$ et $\alpha U \alpha^{-1} \subset U$, et on note $\phi$ l'endomorphisme de $\C[[\widetilde{H}]]$ induit par l'endomorphisme $x \mapsto \alpha x \alpha^{-1}$. L'inclusion $A\hookrightarrow\C[[\widetilde{H}]]$ a ainsi une image stable par $\phi$. Comme $\alpha U\alpha^{-1}\subset U$ est un sous-groupe ouvert d'indice fini, on a $\C[[U]]=\bigoplus_{g\in U/\alpha U\alpha^{-1}}\C[[\alpha U \alpha^{-1}]]g$, ce qui prouve que $\phi$ est un endomorphisme injectif et plat de $A$. On dit qu'un $\C[[\widetilde{H}]][X]_\phi$-module est lisse s'il est localement artinien en tant que $\C[[\widetilde{H}]]$-module.

\begin{exem}
Soit $M=\left( \begin{smallmatrix}\mathcal{O}_\b\backslash\{0\} & \mathcal{O}_\b \\ 0 & 1 \end{smallmatrix}\right) \subset G$. Soit $V$ une représentation du monoïde $M$. Alors $V$ est naturellement un $\C[[\widetilde{H}]][X]_{\phi}$-module de la façon suivante. La structure de $\C[[\widetilde{H}]]$-module provient de l'action de $\widetilde{H}$ sur $V$. L'endomorphisme $\phi_V$ provient de l'action de la matrice $\alpha$. En particulier si $\pi$ est une représentation lisse de $G$, sa restriction à $M$ est un $\C[[\widetilde{H}]][X]_{\phi}$-module lisse.\\
On peut aussi considérer la construction suivante (qui apparaît par exemple dans \cite{Colmezfoncteur} et \cite{HuDiag}). Soit $\pi$ une représentation lisse de $G$ et $\sigma \subset \pi|_\K$ une sous $\K$-représentation de dimension finie sur $\C$. On note $I^+(\pi,\sigma)$ le sous-$A[X]_\phi$-module engendré par $\sigma$. Comme $\sigma$ est stable par $H$ et que $\alpha$ commute à $H$, $I^+(\pi,\sigma)$ est stable par $H$ et a une structure de $\C[[\widetilde{H}]]_{\phi}$-module.
\end{exem}

Le sous-groupe $U$ est stable sous-l'action de $H$ par conjugaison, ainsi $\widetilde{H}$ est isomorphe à un produit semi-direct $U\rtimes H$ et la projection sur le second facteur est un morphisme de groupes $\widetilde{H}\rightarrow H$ dont le noyau est $U$. On obtient ainsi un morphisme continu et surjectif d'anneaux pseudocompacts $\C[[\widetilde{H}]]\rightarrow\C[[H]]$ dont le noyau est l'idéal $\mathfrak{m}_A\C[[\widetilde{H}]]$. Ceci prouve que l'inclusion $\C\hookrightarrow\C[[H]]$ induit un isomorphisme de foncteurs $\C\otimes_A-\xrightarrow{\sim}\C[[H]]\otimes_{\C[[\widetilde{H}]]}-$ sur la catégorie des $\C[[\widetilde{H}]]$-modules. L'isomorphisme $\widetilde{H}\simeq U\rtimes H$ induit un isomorphisme de $\C[[U]]$-modules $\C[[\widetilde{H}]]\simeq\C[[U]]\widehat{\otimes}_\C\C[[H]]$ qui prouve qu'un $\C[[\widetilde{H}]]$-module libre est $\C[[U]]$-plat. Ainsi on a des isomorphismes de foncteurs pour tout $i\geq 0$, $\tor_i^A(\C,-)\simeq\tor_i^{\C[[\widetilde{H}]]}(\C[[H]],-)$, ce qui nous permet de conclure que si $M$ est un $\C[[\widetilde{H}]]$-module, les groupes $\tor_i^A(\C,M)$ sont munis d'une structure de $\C[[H]]_\phi$-module, autrement dit, d'une structure de $\C[[H]]$-module et d'un endomorphisme $\C[[H]]$-linéaire $\phi_{\tor_i^A(\C,M)}$.

\begin{prop}\label{prop:EPG}
Soit $M$ un $\C[[\widetilde{H}]][X]_{\phi}$-module lisse dont le $A[X]_\phi$-module sous-jacent est de présentation finie. Alors $\chi^{A,\phi}(M)=0$.
\end{prop}

\begin{proof}
Soit $\phi'$ l'endomorphisme de $A$ induit par la conjugaison par la matrice $\left(\begin{smallmatrix}p&0\\0&1\end{smallmatrix}\right)$. On a $p\varpi^{-e}\in\mathcal{O}_\b^{\times}$, donc il existe $u\in H$ tel que $[u]\phi^e=\phi'$, où l'on note $[-]$ le morphisme d'inclusion $H\rightarrow\C[[H]]^{\times}$. Soit $h:\,\C[[\widetilde{H}]][X]_{\phi'}\rightarrow \C[[\widetilde{H}]][X]_{\phi}$ le morphisme d'anneaux défini par $h(\sum_ia_iX^i)=\sum_ia_i[u]^iX^{ei}$. Soit $h_*M$ le module image directe de $M$ par $h$, c'est en particulier un $A[X]_{\phi'}$-module. Comme $h|_A=\mathrm{Id}_A$, on a $\tor_i^A(\C,M)=\tor_i^A(\C,h_*M)$ et $\phi'_{\tor_i^A(\C,h_*M)}=[u]\phi_{\tor_i^A(\C,M)}^e$, ce qui prouve que $\mathrm{rg}_{\C[\phi']}(\tor_i^A(\C,h_*M))=\mathrm{rg}_{\C[\phi]}(\tor_i^A(\C,M))$ et donc $\chi^{A,\phi}(M)=\chi^{A,\phi'}(h_*M)$. Soit $(e_i)_{1\leq i\leq d}$ une base du $\mathbb{Z}_p$-module $\mathcal{O}_\b$. On a $\C[[U]]\simeq\C[[X_1,\dots,X_d]]$ où $X_i=[e_i]-1$. On vérifie facilement que $\phi'(X_i)=X_i^{p}$ pour tout $i$. Par ailleurs, comme $M$ est un $A[X]_\phi$-module de présentation finie, $h_*M$ est un $A[X]_{\phi'}$-module de présentation finie, on peut donc appliquer la proposition \ref{prop:EP} à $h_*M$ et conclure que $\chi^{A,\phi'}(h_*M)=0$.
\end{proof}

Soit $(\sigma,V)$ une représentation lisse du groupe $\K$. On note $\ind_\K^G(\sigma)$ la représentation de $G$ dont l'espace sous-jacent est l'ensemble des fonctions $f:\,G\rightarrow V$ telles que, pour tout $g\in G$, $k\in \K$, $f(gk)=\sigma(k^{-1})f(g)$, et dont le support est compact modulo $\K$. L'action de $G$ est la translation à gauche. Il s'agit d'une représentation lisse de $G$. Si $g\in G$ et $v\in\sigma$, on note $[g,v]$ l'unique fonction de support $g\K$ prenant la valeur $v$ en $g$.

\begin{defi}
On dit qu'une représentation lisse $\pi$ de $G$ est de présentation finie s'il existe une représentation $\sigma$ lisse irréductible de $\K$ et une surjection $G$-équivariante $\ind_\K^G(\sigma)\twoheadrightarrow\pi$ dont le noyau est engendré, en tant que $\C[G]$-module, par un nombre fini d'éléments.
\end{defi}

Rappelons le critère suivant dû à Hu.

\begin{theo}[Hu,Thm. $1.3$ de \cite{HuDiag}]\label{theo:Hu}
Soient $\pi$ une représentation de $G$ lisse irréductible ayant un caractère central, et $\sigma \subset \pi$ une sous-$\K$-représentation lisse irréductible de $\pi$. Alors si $\pi$ est de présentation finie, le $\C$-espace vectoriel $I^+(\pi,\sigma)^U$ est de dimension finie.
\end{theo}

La condition $\dim_\C I^+(\pi,\sigma)^U<+\infty$ est équivalente à l'admissibilité du $A$-module $I^+(\pi,\sigma)$. C'est pourquoi le résultat suivant, généralisation partielle de \cite[Prop. $3.2$]{EmCoh}, nous sera particulièrement utile.

\begin{lemm}\label{lemm:pf}
Soit $M$ un $\C[[\widetilde{H}]][X]_{\phi}$-module lisse. Supposons que le $A[X]_\phi$-module sous-jacent soit de type fini et admissible. Alors $M$ est un $A[X]_\phi$-module de présentation finie.
\end{lemm}

\begin{proof}
Le morphisme $\phi_M$ induit un morphisme de $A$-modules $\phi^*M \rightarrow M$. Si l'on munit $\phi^*M$ de l'action de $H$ définie par $\gamma(a \otimes m)=\gamma(a) \otimes \gamma(m)$, on vérifie facilement que $\phi_M$ est un morphisme de $A$-modules $H$-équivariant. L'endomorphisme $\phi$ de $A$ est injectif, il se prolonge donc en un endomorphisme $\phi$ du corps des fractions $\mathrm{Frac}(A)$ de $A$. Ainsi $\mathrm{Frac}(A)\otimes_A\phi^*M=\mathrm{Frac}(A)\otimes_{\mathrm{Frac}(A),\phi}(\mathrm{Frac}(A)\otimes_AM)$. Comme $M$ est admissible, les deux $A$-modules $\phi^*M$ et $M$ ont même corang fini. Comme $M$ est de type fini sur $A[X]_\phi$, $\mathrm{coker}(\phi_M)$ est un $A$-module localement artinien de type fini, donc de longueur finie. Ainsi l'application $\phi_M^{\vee} : \, M^{\vee} \rightarrow (\phi^*M)^{\vee}$ a un noyau de $A$-torsion, donc un conoyau de $A$-torsion. Ainsi $\ker(\phi_M)^{\vee}\simeq\mathrm{coker}(\phi_M^{\vee})$ est un $A$-module de torsion et de type fini. Comme $\ker(\phi_M)^{\vee}$ est muni d'une action semi-linéaire de $H$, le corollaire $4.3$ de \cite{HMS} montre que $\ker(\phi_M)^{\vee}$ est un $A$-module de longueur finie, donc $\ker(\phi_M)$ également. Par \cite[Lemma $1.2$]{EmCoh}, le $A[X]_\phi$-module $M$ est de présentation finie.
\end{proof}

\begin{rema}
Le résultat de \cite{HMS} utilisé ici est également une conséquence d'un résultat récent de K.~Ardakov (\cite[Corollary $8.1$]{ArdZal}).
\end{rema}

\begin{coro}\label{coro:nonadm}
Supposons que $[F:\mathbb{Q}_p]=2$. Soit $\mathcal{M}$ un $\C[[\widetilde{H}]][X]_{\phi}$-module lisse de présentation finie tel que $h_0(\mathcal{M})=0$. S'il existe un $\C[[\widetilde{H}]]A[X]_{\phi}$-module quotient $M$ de $\mathcal{M}$ admissible et tel que le noyau de $\tor_2^A(\C,\mathcal{M}) \rightarrow \tor_2^A(\C,M)$ soit sans torsion, alors $\mathcal{M}$ est admissible.
\end{coro}

\begin{proof}
Soit $N$ le noyau de $\mathcal{M} \rightarrow M$. D'après le lemme \ref{lemm:pf}, $M$ est un $A[X]_\phi$-module de présentation finie, donc par cohérence de $A[X]_\phi$, le $A[X]_\phi$-module $N$ est aussi de présentation finie. Comme $M$ est admissible, le corollaire \ref{coro:adm} montre que $h_1(M)=0$ et $\chi^{A,\phi}(M)=0$, donc $h_0(N)=h_0(\mathcal{M})=0$. De plus la proposition \ref{prop:EPG} montre que $\chi^{A,\phi}(\mathcal{M})=0$ et donc $\chi^{A,\phi}(N)=0$. Soit $\widetilde{N} \subset N$ le sous-$A[X]_\phi$-module construit dans la proposition \ref{prop:admnonpf}. On a $\tor_2^A(\C,N/ \widetilde{N})=\tor_2^A(\C,N)_{tors}=\ker(\tor_2^A(\C,\mathcal{M})\rightarrow\tor_2^A(\C,M))_{tors}=0$, et donc le lemme \ref{lemm:inj} implique $N=\widetilde{N}$. Le lemme \ref{lemm:pf} montre que $M$ est de présentation finie, donc par cohérence de $A[X]_\phi$, $N$ est de présentation finie. La proposition \ref{prop:admnonpf} montre que $N$ est admissible. Dans ce cas, $\mathcal{M}$ est une extension de $A$-modules admissibles, donc est admissible.
\end{proof}

Notons la réciproque partielle au théorème \ref{theo:Hu} qui cependant ne nous servira pas dans la suite.

\begin{lemm}\label{lemm:I1fin}
Soit $V$ une représentation admissible de $I_1$ telle que $V^{U\cap I_1}$ et $V^{U^-\cap I_1}$ soient tous les deux de dimension finie. Alors $V$ est un $\C$-espace vectoriel de dimension finie.
\end{lemm}

\begin{proof}
Soit $M=V^{\vee}$ son dual de Pontryagin. C'est un $\C[[I_1]]$-module de type fini. Soit $N$ un quotient cyclique de $M$. La $I_1$-représentation $W=N^{\vee}$ vérifie les mêmes hypothèses que $V$. Considérons une inclusion $I_1$-équivariante de $W$ dans l'espace $C^{\infty}(I_1,\C)$ des fonctions localement constantes de $I_1$ dans $\C$, muni de l'action de $I_1$ par translation à gauche. La décomposition d'Iwahori nous donne un isomorphisme $U\cap I_1$-équivariant $C^{\infty}(I_1,\C) \simeq C^{\infty}(U\cap I_1,\C) \otimes_\C C^{\infty}(U^-\cap I_1,\C)$. Comme $W^{U \cap I_1}$ est de dimension finie, il existe un sous-$\C$-espace vectoriel $X \subset C^{\infty}(U^- \cap I_1,\C)$ de dimension finie tel que $W \subset C^{\infty}(U \cap I_1, \C) \otimes_\C X$. Si $H_1 \subset U^- \cap I_1$ est un sous-groupe ouvert tel que $X \subset C^{\infty}((U^- \cap I_1)/H_1,\C)$, l'inclusion précédente implique que $W$ est contenu dans le sous-espace de $C^{\infty}(I_1,\C)$ des fonctions invariantes à droite par $H_1$. De même le fait que $W^{U^- \cap I_1}$ est de dimension finie implique qu'il existe $H_2 \subset U \cap I_1$ sous-groupe ouvert tel que $W$ est aussi contenu dans le sous-espace de $C^{\infty}(I_1,\C)$ constitué des fonctions invariantes par translation à droite par $H_2$. Comme $H_1$ et $H_2$ engendrent un sous-groupe ouvert de $I_1$, on voit que $W$ est de dimension finie. Pour conclure, tous les quotients cycliques de $M$ sont de dimension finie sur $\C$. Comme $M$ est de type fini sur $\C[[I_1]]$, on voit que $M$ est de dimension finie sur $\C$.
\end{proof}

\begin{coro}
Soient $\pi$ une représentation lisse de $G$ et $\sigma \subset \pi$ une sous-$\K$-représentation lisse irréductible de $\pi$. Supposons que $I^+(\pi,\sigma)^U$ soit de dimension finie. Alors $\pi$ est de présentation finie.
\end{coro}

\begin{proof}
On utilise les résultats et notations de \cite{HuDiag}. En remarquant que $\beta(I^+(\pi,\sigma)^U)=I^-(\pi,\sigma)^{U^-\cap I_1}$, on voit que $I^-(\pi,\sigma)^{U^-\cap I_1}$ est de dimension finie. En conséquence, $D_1(\pi,\sigma)=I^+(\pi,\sigma)\cap I^-(\pi,\sigma)$ vérifie les hypothèses du lemme \ref{lemm:I1fin} et est par conséquent de dimension finie. Par \cite[Thm. $1.3$]{HuDiag}, $\pi$ est de présentation finie.
\end{proof}

\subsection{Représentations supersingulières}

Soit $G=\coprod_{m\in\mathbb{Z}}I\alpha^m\K$ la décomposition de Cartan-Iwahori de $G$. Pour $i\in\mathbb{Z}$, notons $R_i(\sigma)$ le sous-espace de $\ind_\K^G(\sigma)$ constitué des fonctions à support dans $I\alpha^i\K$ et $I_{\geq i}(\sigma)=\bigoplus_{m\geq i}R_m(\sigma)$. Pour $m\geq0$, on a $\alpha I\alpha^m\K\subset I\alpha^{m+1}\K$, donc $I_{\geq i}(\sigma)$ est un sous-$A[X]_\phi$-module de $\ind_\K^G(\sigma)$ pour $i\geq0$.

Soit $\sigma$ une représentation lisse irréductible de $\K$. Il existe un isomorphisme canonique de $\C$-espaces vectoriels entre $\mathrm{End}_G(\ind_\K^G(\sigma))$ et l'espace des fonctions $f:\,G\rightarrow\mathrm{End}_\C(\sigma)$ vérifiant, pour tout $(k_1,g,k_2)\in\K\times G\times\K$, $f(k_1gk_2)=\sigma(k_1)\circ f(g)\circ\sigma(k_2)$ et à support compact modulo $\K$. L'existence de cet isomorphisme est prouvé par Barthel et Livné dans \cite[\S$2$]{BLmodular}. En utilisant la classification des représentations lisses irréductibles de $\K$, ils prouvent également qu'il existe un isomorphisme d'algèbres $\mathrm{End}_G(\ind_{KZ}^G(\sigma))\simeq\C[T]$ et que l'on peut choisir l'endomorphisme $T$ correspondant à une fonction $f:\,G\rightarrow\mathrm{End}_\C(\sigma)$ à support dans $\K\alpha\K$ sous l'isomorphisme précédent. On pose alors $\pi(\sigma,0)=\ind_\K^G(\sigma)/T(\ind_\K^G(\sigma))$. Suivant la terminologie de \cite{BLmodular}, on dit qu'une représentation $\pi$ de $G$ lisse et irréductible est supersingulière s'il existe une représentation lisse irréductible $\sigma$ de $\K$ telle que $\pi$ est isomorphe à un quotient de $\pi(\sigma,0)$. La classification des représentations non supersingulières fait l'objet du théorème $34$ de \cite{BLmodular}.\\

Soit $\pi$ une représentation lisse irréductible supersingulière de $G$, $\sigma$ une représentation lisse irréductible de $\K$ et $\varphi:\,\pi(\sigma,0)\twoheadrightarrow\pi$ une surjection $G$-équivariante. Identifions $\sigma$ à la sous-$U$-représentation de $\ind_\K^G(\sigma)$ par l'isomorphisme $v\mapsto [1,v]$. Comme $\sigma$ engendre $\ind_\K^G(\sigma)$ en tant que représentation de $G$, la restriction de $\varphi$ à $\sigma$ est injective, et donc nous pouvons considérer, via $\varphi$, $\sigma$ comme une sous-représentation de $\pi$. Le sous-espace vectoriel $I_{\geq0}(\sigma)\subset\ind_\K^G(\sigma)$ est le sous-$\C$-espace vectoriel engendré par le monoïde $\left(\begin{smallmatrix}\mathcal{O}_\b\backslash\{0\}&\mathcal{O}_\b\\0&1\end{smallmatrix}\right)$ et $\sigma$, donc $I^+(\pi,\sigma)\subset\pi$ est par définition l'image de $I_{\geq0}(\sigma)$ par $\varphi$. Les formules de la proposition $5$ de \cite{BLmodular} montrent que si $f\in R_m(\sigma)$, alors $T(f)\in R_{m-1}(\sigma)\oplus R_{m+1}(\sigma)$, en particulier $T(I_{\geq1}(\sigma))\subset I_{\geq0}(\sigma)$. L'application $\varphi$ se factorise donc en une surjection de $A[X]_\phi$-modules $I_{\geq0}(\sigma)/(T(I_{\geq1}(\sigma)))\twoheadrightarrow I^+(\pi,\sigma)$.

\begin{lemm}\label{lemm:iso}
L'injection $\sigma\rightarrow I_{\geq0}(\sigma)$ envoyant $v\in\sigma$ sur $[1,v]$ induit un isomorphisme de $A[X]_\phi$-modules $A[X]_\phi\otimes_A\sigma\xrightarrow{\sim}I_{\geq0}(\sigma)$. Pour tout $m\geq0$, cet isomorphisme identifie alors le sous-$A$-module $A X^m\otimes_A\sigma$ de $A[X]_\phi\otimes_A\sigma$ à $R_m(\sigma)$ et le sous-$A[X]_\phi$-module $A[X]_\phi X^m(A\otimes_A\sigma)\subset A[X]_\phi\otimes_A\sigma$ à $I_{\geq m}(\sigma)$.
\end{lemm}

\begin{proof}
Soit $m\geq0$. Le sous-espace $\alpha^m(R_0(\sigma))$ de $\ind_\K^G(\sigma)$ coïncide avec l'ensemble des fonctions à support dans $(\alpha^mU\alpha^{-m})\alpha^m\K$. Comme $I\alpha^m\K=U\alpha^m\K=\coprod_{u\in U/\alpha^mU\alpha^{-m}}u\alpha^m\K$, on a un isomorphisme de $U$-modules entre $R_m(\sigma)$ et $\ind_{\alpha^mU\alpha^{-m}}^U(\alpha^m(R_0(\sigma)))$. Ceci prouve que l'application $h:\,A[X]_\phi\otimes_A\sigma\rightarrow I_{\geq0}(\sigma)$ est surjective et envoie $A X^m\otimes_A \sigma\simeq(\phi^*)^m\sigma$ sur $R_m(\sigma)$. Comme ces deux $\C$-espaces vectoriels ont même dimension, la restriction de $h$ à $A X^m\otimes_A \sigma$ est injective, donc $h$ également. La deuxième assertion est une conséquence immédiate de ce qui précède.
\end{proof}

Une conséquence de résultat est que le $A[X]_\phi$-module $I_{\geq0}(\sigma)/T(I_{\geq1}(\sigma))$ est de présentation finie.

Comme $T(R_m(\sigma))\subset R_{m-1}(\sigma)\oplus R_{m+1}(\sigma)$, l'opérateur $T$ peut s'écrire de façon unique comme somme de deux opérateurs $T_+$ et $T_-$ tels que pour tout $m\in\mathbb{Z}$, $T_\pm(R_{m}(\sigma))\subset R_{m\pm1}(\sigma)$. L'opérateur $T$ étant $G$-équivariant, il est clair que les opérateurs $T_+$ et $T_-$ sont $I$-équivariants, en particulier $U$-équivariants.

\begin{lemm}\label{lemm:surjinj}
Pour tout $m\geq1$, l'opérateur $T_-:\,R_{m}(\sigma)\rightarrow R_{m-1}(\sigma)$ est surjectif mais non injectif et l'opérateur $T_+:\,R_m(\sigma)\rightarrow R_{m+1}(\sigma)$ est injectif. De plus l'application $\tor_d^A(T_+)$ (resp. $\tor_0^A(T_-)$) est un isomorphisme $\tor_d^A(\C,R_m(\sigma))\simeq\tor_d^A(\C,R_{m+1}(\sigma))$ (resp. un isomorphisme $\tor_0^A(\C,R_m(\sigma))\simeq\tor_0^A(\C,R_{m-1}(\sigma))$), alors que les applications $\tor_d^A(T_-)$ et $\tor_0^A(T_+)$ sont nulles.
\end{lemm}

\begin{proof}
Par le lemme $2$ et la proposition $4$ de \cite{BLmodular}, on a $\dim_\C\sigma^U=\dim_\C\sigma_U=1$.\\
L'isomorphisme du lemme \ref{lemm:iso} associé à l'isomorphisme $A[X]_\phi\otimes_A\sigma\simeq\bigoplus_{m\geq0}(\phi^*)^m\sigma$ induit des isomorphismes de $A$-modules $(\phi^*)^m\sigma\simeq R_m(\sigma)$. Comme l'opérateur $T$ est $G$-équivariant, en identifiant $(\phi^*)^m\sigma$ à $R_m(\sigma)$ via ces isomorphismes, on a $(\phi^*)^{m-1}(T_\pm|_{R_1(\sigma)})=T_\pm|_{R_m(\sigma)}$ pour tout $m\geq1$. Ainsi il suffit de prouver le résultat pour $m=1$ et d'utiliser la platitude du foncteur $\phi^*$ et le diagramme \eqref{eq:diagTor}. Dans le cas $m=1$, la surjectivité de $T_-$ est alors conséquence de la formule explicite suivante : pour $v\in\sigma^U$, $T_-([\alpha,w(v)])=[1,\chi(\varpi)w(v)]$, où $\chi$ désigne le caractère définissant l'action de $Z$ sur $\sigma$, et du fait que $w(\sigma^U)$ engendre $\sigma$ comme $A$-module. Comme $\dim_\C\sigma_U=\dim_\C\phi^*(\sigma)_U=1$, l'application $\tor_0^A(T_-|_{\phi^*\sigma})$ est bijective. Or le $\C$-espace vectoriel $\phi^*\sigma$ est de dimension $p^f\dim_\C\sigma$, l'application $T_-$ n'est donc pas injective. En particulier, $T_-|_{(\phi^*\sigma)^U}$ n'est pas injective. Comme $\dim_\C\sigma^U=1$, on a $\dim_\C \phi^*\sigma^U=1$ et donc $T_-(\phi^*\sigma^U)=0$. Comme l'opérateur $T$ est injectif (\cite[Thm. $19$]{BLmodular}), on doit avoir $T_+|_{(\phi^*\sigma)^U}$ injectif, c'est-à-dire, par application du lemme \ref{lemm:inj}, $T_+$ injectif. On en déduit les assertions sur $\tor_d^A$. De même, $T_+$ ne peut être un isomorphisme, donc $T_+$ n'est pas surjectif, ce qui implique $\tor_0^A(T_+)=0$.
\end{proof}

\subsection{Non admissibilité du conoyau de $T$}

Cette section concerne la structure de $\C[X]$-module de $\tor_{[\b:\mathbb{Q}_p]}^A(\C,I_{\geq0}(\sigma)/T(I_{\geq1}(\sigma)))$. Par commodité, nous supposons ici $[\b:\mathbb{Q}_p]=2$.

\subsubsection{Le cas non ramifié}

Supposons ici $\b$ extension quadratique non ramifiée de $\mathbb{Q}_p$. Si $u\in\res$, on note $[u]\in\mathcal{O}_\b$ le représentant de Teichmüller de $u$. On pose alors
\begin{equation*}
X=\sum_{\lambda\in\res}\lambda^{-1}\left(\begin{matrix}1&[\lambda]\\0&1\end{matrix}\right)\in A,\ Y=\sum_{\lambda\in\res}\lambda^{-p}\left(\begin{matrix}1&[\lambda]\\0&1\end{matrix}\right)\in A.
\end{equation*}

\begin{prop}
Les éléments $X$ et $Y$ forment un système régulier de paramètres de l'idéal maximal $\mathfrak{m}$ de $\C[[U]]$, autrement dit, on a un isomorphisme $\C[[X,Y]]\simeq\C[[U]]$.
\end{prop}

\begin{proof}
Il suffit de montrer que les images $\overline{X}$ et $\overline{Y}$ de $X$ et $Y$ dans $\mathfrak{m}/\mathfrak{m}^2$ forment une $\C$-base de cet espace. Fixons donc $\alpha\in\res$ un élément primitif, c'est-à-dire tel que $(\alpha,\alpha^p)$ soit une $\mathbb{F}_p$-base de $\res$. Si $a\in\res$, posons $u_a=\left(\begin{matrix}1&[a]\\0&1\end{matrix}\right)-1$. On a alors $u_{a^i}\equiv i u_a \, \mod \mathfrak{m}^2$ pour tout $a\in\res$, d'où
\begin{align*}
X&\equiv\sum_{(i,j)\in\mathbb{F}_p^2}(i\alpha+j\alpha^p)^{q-2}(iu_{\alpha}+ju_{\alpha^p})\,\mod\mathfrak{m}^2\\
&\equiv au_{\alpha}+bu_{\alpha^p}\,\mod\mathfrak{m}^2
\end{align*}
où l'on a posé $a=\sum_{(i,j)\in\mathbb{F}_p^2}(i\alpha+j\alpha^p)^{q-2}i$ et $b=\sum_{(i,j)\in\mathbb{F}_p^2}(i\alpha+j\alpha^p)^{q-2}j$. De même, on a
\begin{equation*}
Y\equiv bu_{\alpha}+au_{\alpha^p}\,\mod\mathfrak{m}^2.
\end{equation*}
Comme $([\alpha],[\alpha^p])$ est une base du $\mathbb{Z}_p$-module $\mathcal{O}_\b$, $(u_{\alpha},u_{\alpha^p})$ forme une base du $\C$-espace vectoriel $\mathfrak{m}/\mathfrak{m}^2$. Pour prouver que $(\overline{X},\overline{Y})$ est une base de cet espace, il suffit de prouver que la matrice $\left(\begin{smallmatrix}a&b\\b&a\end{smallmatrix}\right)$ est inversible. Elle est de déterminant $a^2-b^2$, il suffit donc de vérifier que $a\neq\pm b$. Par ailleurs, un calcul facile montre que $a\alpha+b\alpha^p=-1$ et $a\alpha^p+b\alpha=0$, on en déduit que $a\neq\pm b$ et donc le résultat.
\end{proof}

Soit $r$ un entier compris entre $0$ et $p-1$. Le groupe $\mathrm{GL}_2(\res)$ agit naturellement sur l'espace $\C^2$, et donc sur le produit symétrique $\mathrm{Sym}^r(\C^2)$. Pour $g\in\mathrm{GL}_2(\res)$, notons $\rho_r(g)$ l'endomorphisme de $\mathrm{Sym}^r(\C^2)$ ainsi obtenu. On note alors $\mathrm{Sym}^r(\C^2)^{\mathrm{Fr}}$ la représentation de $\mathrm{GL}_2(\res)$ dont l'espace vectoriel sous-jacent est $\mathrm{Sym}^r(\C^2)$ mais pour la quelle l'action de $g\in\mathrm{GL}_2(\res)$ est donnée par $\rho_r(\mathrm{Fr}(g))$, $\mathrm{Fr}$ étant l'automorphisme de $\mathrm{GL}_2(\res)$ induit par l'automorphisme de Frobenius de $\res$. Si $\Vec{r}=(r_0,r_1)$ est un couple d'entiers compris entre $0$ et $p-1$, on note $\mathrm{Sym}^{\Vec{r}}(\C^2)$ la représentation de $\mathrm{GL}_2(\res)$ sur $\mathrm{Sym}^{r_0}(\C^2)\otimes_\C\mathrm{Sym}^{r_1}(\C^2)^{\mathrm{Fr}}$. D'après \cite[Prop. $1$]{BLmodular}, une telle représentation est irréductible.

\begin{prop}\label{prop:gennr}
Soient $\Vec{r}=(r_0,r_1)$ un couple d'entiers compris entre $0$ et $p-1$ et $\mathrm{Sym}^{\Vec{r}}(\C^2)$ la représentation irréductible de $\mathrm{GL}_2(\res)$ définie ci-dessus. Soit $v\in\mathrm{Sym}^{\Vec{r}}(\C^2)^U$ un vecteur non nul. Alors $w(v)$ engendre $\mathrm{Sym}^{\Vec{r}}(\C^2)$ comme $A$-module et l'annulateur de $w(v)$ est l'idéal $(X^{r_0+1},Y^{r_1+1})$.
\end{prop}

\begin{proof}
Le premier point est une conséquence de \cite[Lemma $2$]{BLmodular}. Un calcul montre que $X^r\equiv\sum_{\lambda\in\res}\lambda^{-r}u_{\lambda}\,\mod (X^p,Y^p)$ et $Y^r\equiv\sum_{\lambda\in\res}\lambda^{-pr}u_{\lambda}\,\mod (X^p,Y^p)$ pour $1\leq r \leq p-1$. On remarque alors que les éléments $X^p$ et $Y^p$ annulent $\mathrm{Sym}^{\Vec{r}}(\C^2)$ et que pour $0\leq r_0\leq p-2$ (resp. $0\leq r_1\leq p-2$), l'élément $\sum_{\lambda\in\res}\lambda^{-(r_0+1)}u_{\lambda}$ (resp. $\sum_{\lambda\in\res}\lambda^{-p(r_1+1)}u_{\lambda}$) annule $w(v)$, donc au final l'idéal $(X^{r_0+1},Y^{r_1+1})$ également, même si $r_0=p-1$ ou $r_1=p-1$. On conclut en remarquant que $\dim_\C\mathrm{Sym}^{\Vec{r}}(\C^2)=\dim_\C A/(X^{r_0+1},Y^{r_1+1})=(r_0+1)(r_1+1)$.
\end{proof}

\begin{prop}\label{prop:calnr}
Soient $0\leq r_0,r_1 \leq p-1$ et $I=(X^{r_0+1},Y^{r_1+1})$. Soient $s_1:\,A/\phi(I)\rightarrow (A/I)$ et $s_2:\,A/\phi(I)\rightarrow A/\phi^2(I)$ deux morphismes de $A$-modules, avec $s_2$ injectif, alors les noyaux des applications $\tor_1^A(s_1)$ et $\tor_1^A(s_2)$ ont un élément non nul en commun.
\end{prop}

\begin{proof}
Les application $s_1$ et $s_2$ se relèvent en des endomorphismes $\tilde{s_1}$ et $\tilde{s_2}$ de $A$ tels que $\tilde{s}_1(\phi(I))\subset I$ et $\tilde{s}_2(\phi(I))\subset\phi^2(I)$. Notons $\overline{s}_1:\,\C\otimes_A\phi(I)\rightarrow\C\otimes_AI$ et $\overline{s}_2:\,\C\otimes_A\phi(I)\rightarrow\C\otimes_A\phi^2(I)$ les applications obtenues en restreignant $\tilde{s}_1$ et $\tilde{s}_2$ à $\phi(I)$ puis en réduisant modulo $\mathfrak{m}$. Les suites exactes $0\rightarrow\phi^m(I)\rightarrow A\rightarrow A/\phi^m(I)\rightarrow0$ donnent des isomorphismes $\tor_1^A(\C,A/\phi^m(I))\simeq\C\otimes_A\phi^m(I)$ transformant $\tor_1^A(s_i)$ en $\overline{s}_i$. Il suffit donc de montrer que les noyaux de $\overline{s}_1$ et $\overline{s}_2$ ont un élément non nul en commun.\\
Remarquons que $\phi(X)=Y^p$ et $\phi(Y)=X^p$, donc $\phi(I)=(X^{p(r_1+1)},Y^{p(r_0+1)})$ et $\phi^2(I)=(X^{p^2(r_0+1)},Y^{p^2(r_1+1)})$. Les morphismes $\tilde{s}_1$  et $\tilde{s}_2$ peuvent s'écrire sous la forme $b\mapsto ba_i$, avec $a_1,a_2\in A$ avec les conditions $X^{p(r_1+1)}a_2\in (X^{p^2(r_0+1)},Y^{p^2(r_1+1)})$ et $Y^{p(r_0+1)}a_2\in (X^{p^2(r_0+1)},Y^{p^2(r_1+1)})$. De plus, le sous-espace de $A/\phi(I)$ (resp. $A/\phi^2(I))$) constitué des éléments annulés par $\mathfrak{m}$ est de dimension $1$ et engendré par $X^{p(r_1+1)-1}Y^{p(r_0+1)-1}$ (resp. $X^{p^2(r_0+1)-1}Y^{p^2(r_1+1)-1}$). Comme $s_2$ est injectif, on a $a_2X^{p(r_1+1)-1}Y^{p(r_0+1)-1}\equiv cX^{p^2(r_0+1)-1}Y^{p^2(r_1+1)-1}\,\mod \phi^2(I)$, pour un certain $c\in\C^{\times}$. On peut donc écrire $a_2X^{p(r_1+1)-1}Y^{p(r_0+1)-1}=cX^{p^2(r_0+1)-1}Y^{p^2(r_1+1)-1}+P_1(X,Y)X^{p^2(r_0+1)}+P_2(X,Y)Y^{p^2(r_1+1)}$. Comme $p(r_1+1)-1\leq p^2(r_0+1)-1$ et $p(r_0+1)-1\leq p^2(r_1+1)-1$, on a $Y^{p(r_0+1)-1}|P_1(X,Y)$ et $X^{p(r_1+1)-1}|P_2(X,Y)$. On en déduit $a_2\equiv cX^{p^2(r_0+1)-p(r_1+1)}Y^{p^2(r_1+1)-p(r_0+1)}\,\mod(X^{p^2(r_0+1)},Y^{p^2(r_1+1)})$. On peut donc choisir $\tilde{s}_2$ tel que $a_2=cX^{p^2(r_0+1)-p(r_1+1)}Y^{p^2(r_1+1)-p(r_0+1)}$.\\
Comme $\C\otimes_A \phi^2(I)$ est un $\C$-espace vectoriel dont une base est donnée par les images de $X^{p^2(r_0+1)}$ et $Y^{p^2(r_1+1)}$, l'application $\overline{s}_2$ est non nulle si et seulement si $p^2(r_0+1)=p(r_1+1)$ ou $p^2(r_1+1)=p(r_0+1)$, c'est-à-dire $(r_0,r_1)=(p-1,0)$ ou $(r_0,r_1)=(0,p-1)$. Dans le premier cas, le noyau de $\overline{s}_2$ est engendré par la classe de $Y^{p(r_0+1)}=Y^{p^2}$, mais on vérifie que cet élément est aussi dans le noyau de $\overline{s}_1$. De même, dans le second cas, le noyau de $\overline{s}_2$ est engendré par la classe de $X^{p(r_1+1)}=X^{p^2}$, qui est dans le noyau de $\overline{s}_1$. Enfin, si $(r_0,r_1)\notin\{(p-1,0),(0,p-1)\}$ les applications $\overline{s}_1$ et $\overline{s}_2$ sont nulles toutes les deux.
\end{proof}

\subsubsection{Le cas totalement ramifié}

Nous considérons à présent le cas où $\b$ est une extension quadratique totalement ramifiée de $\mathbb{Q}_p$. On pose alors
\begin{equation*}
X=\sum_{\lambda\in\res}\lambda^{-1}\left(\begin{matrix}1&[\lambda]\\0&1\end{matrix}\right)\in A,\ Y=\sum_{\lambda\in\res}\lambda^{-1}\left(\begin{matrix}1&\varpi[\lambda]\\0&1\end{matrix}\right)\in A.
\end{equation*}

\begin{prop}
Les éléments $X$ et $Y$ forment un système régulier de paramètres de l'idéal maximal $\mathfrak{m}$ de $\C[[U]]$, autrement dit, on a un isomorphisme $\C[[X,Y]]\simeq\C[[U]]$.
\end{prop}

\begin{proof}
Remarquons tout d'abord comme dans le cas non ramifié que l'élément $X$ est un générateur de l'idéal maximal de $\C[[\mathbb{Z}_p]]$. Ainsi on en déduit que $Y$ est un générateur de l'idéal maximal de $\C[[\varpi\mathbb{Z}_p]]$. On utilise alors l'isomorphisme $\C[[\mathbb{Z}_p]]\widehat{\otimes}_{\C}\C[[\varpi\mathbb{Z}_p]]\xrightarrow{\sim}\C[[\mathcal{O}_\b]]$ provenant de $\mathcal{O}_\b=\mathbb{Z}_p\oplus\varpi\mathbb{Z}_p$.
\end{proof}

Soit $0\leq r \leq p-1$. Par inflation, l'action du groupe $\mathrm{GL}_2(\res)$ sur le $\C$-espace vectoriel $\mathrm{Sym}^r(\C^2)$ donne une représentation lisse irréductible du groupe $\mathrm{GL}_2(\mathcal{O}_\b)$ que l'on désigne par le même symbole.

\begin{prop}\label{prop:genram}
Soit $0\leq r \leq p-1$. Soit $v\in\mathrm{Sym}^{\Vec{r}}(\C^2)^U$ un vecteur non nul. Alors $w(v)$ engendre $\mathrm{Sym}^{r}(\C^2)$ comme $A$-module et l'annulateur de $w(v)$ est l'idéal $(X^{r+1},Y)$.
\end{prop}

\begin{proof}
On raisonne comme dans le cas non ramifié, mais c'est ici beaucoup plus simple vu que $\res=\mathbb{F}_p$.
\end{proof}

\begin{lemm}
Pour tout $n\geq0$, $0\leq r \leq p-1$, on a $\phi^{2n}((X^{r+1},Y))=(X^{p^n(r+1)},Y^{p^n})$ et $\phi^{2n+1}((X^{r+1},Y))=(X^{p^{n+1}},Y^{p^n(r+1)})$.
\end{lemm}

\begin{proof}
On a $\varpi^2=up$ pour $u\in\mathcal{O}_\b^{\times}$, ainsi on a $\phi(X)=Y$ et $\phi(Y)=u\cdot X^p$. L'idéal $(X^{r+1},Y)$ annulant la représentation $\mathrm{Sym}^{\Vec{r}}(\C)$, il est stable sous l'action de $\mathcal{O}_\b^{\times}$, et comme l'action de $\mathcal{O}_\b^{\times}$ commute à celle de $\phi$, il en est de même pour tous les $\phi^n((X^{r+1},Y))$. Ainsi $(X^p,Y^{r+1})\subset \phi((X^{r+1},Y))$. En comparant l'indice de ces deux idéaux on voit qu'il y a égalité. Le cas général se traite de la même façon.
\end{proof}

\begin{prop}\label{prop:calram}
Soient $0\leq r \leq p-1$ et $I=(X^{r+1},Y)$. Soient $s_1:\,A/\phi(I)\rightarrow (A/I)$ et $s_2:\,A/\phi(I)\rightarrow A/\phi^2(I)$ deux morphismes de $A$-modules, avec $s_2$ injectif, alors les noyaux des applications $\tor_1^A(s_1)$ et $\tor_1^A(s_2)$ ont un élément non nul en commun.
\end{prop}

\begin{proof}
Comme dans le cas non ramifié, on relève $s_1$ et $s_2$ en endomorphismes $\tilde{s}_1$ et $\tilde{s}_2$ de $A$ que l'on peut écrire sous la forme $b\mapsto a_i$, avec $a_1,a_2\in A$ tels que $X^pa_2\in (X^{p(r+1)},Y^p)$ et $Y^{r+1}a_2\in (X^{p(r+1)},Y^p)$. De plus, comme $s_2$ est injectif, il existe $c\in\C^{\times}$ tel que $a_2X^{p-1}Y^r\equiv cX^{p(r+1)-1}Y^{p-1}\,\mod \phi^2(I)$. On peut donc écrire $a_2X^{p-1}Y^r=cX^{p(r+1)-1}Y^{p-1}+P_1(X,Y)X^{p(r+1)}+P_2(X,Y)Y^p$. Ainsi $Y^r|P_1(X,Y)$ et $X^{p-1}|P_2(X,Y)$. On en déduit $a_2\equiv cX^{pr}Y^{p-1-r}\,\mod(X^{p(r+1)},Y^p)$. On peut donc choisir $\tilde{s}_2$ tel que $a_2=cX^{pr}Y^{p-1-r}$.\\
Commençons par traiter le cas $0\leq r<p-1$. Alors $X^p\in\mathfrak{m}(X^{r+1},Y)$ et $X^pa_2=cX^{p(r+1)}Y^{p-1-r}\in\mathfrak{m}(X^{p(r+1)},Y^p)$, c'est-à-dire $\overline{s}_1(X^p)=\overline{s}_2(X^p)=0$.\\
Dans le cas où $0<r\leq p-1$, on a $Y^{r+1}\in\mathfrak{m}(X^{r+1},Y)$ et $Y^{r+1}a_2=cX^{pr}Y^p\in\mathfrak{m}(X^{p(r+1)},Y^p)$, ce qui signifie $\overline{s}_1(Y^{r+1})=\overline{s}_2(Y^{r+1})=0$.\\
Dans tous les cas $\overline{s}_1$ et $\overline{s}_2$ ont un noyau commun, les deux applications sont mêmes nulles si $0<r<p-1$.
\end{proof}

\subsubsection{Non admissibilité du conoyau de $T$}

On déduit des calculs précédents la proposition suivante.

\begin{prop}\label{prop:nonadm}
Le $\C[X]$-module $\tor_2^A(\C,I_{\geq0}(\sigma)/T(I_{\geq1}(\sigma)))$ n'est pas un $\C[X]$-module de torsion, c'est donc $\C$-espace vectoriel de dimension infinie.
\end{prop}

\begin{proof}
La suite exacte longue associée à la suite exacte $0\rightarrow I_{\geq1}(\sigma)\rightarrow I_{\geq0}(\sigma)\rightarrow I_{\geq0}(\sigma)/T(I_{\geq1}(\sigma))\rightarrow0$ s'écrit
\begin{multline*}
0\rightarrow \tor_2^A(\C,I_{\geq1}(\sigma))\rightarrow\tor_2^A(\C,I_{\geq0}(\sigma))\rightarrow\tor_2^A(\C,I_{\geq0}(\sigma)/T(I_{\geq1}(\sigma)))\\\rightarrow\tor_1^A(\C,I_{\geq1}(\sigma))\rightarrow\tor_1^A(\C,I_{\geq0}(\sigma))
\end{multline*}
Comme le $\C[X]$-module $\tor_1^A(\C,I_{\geq1}(\sigma))$ est isomorphe à $\C[X]\otimes_\C\tor_1^A(\C,\phi^*\sigma)$, c'est un $\C[X]$-module libre de type fini, il suffit donc de prouver que la flèche de droite n'est pas injective, et pour cela il suffit de prouver que la restriction de la flèche de droite à $\tor_1^A(\C,\phi^*\sigma)$ n'est pas injective, c'est alors une conséquence des calculs précédents, puisque d'après \cite[Prop. $4$]{BLmodular} et les propositions \ref{prop:gennr} et \ref{prop:genram}, le $A$-module $\sigma$ est toujours de la forme considérée dans les proposition \ref{prop:calnr} et \ref{prop:calram}.
\end{proof}

Nous en déduisons le résultat suivant.

\begin{coro}
Soit $\sigma$ une représentation irréductible lisse de $\K$. La représentation $\pi(\sigma,0)$ n'est pas admissible.
\end{coro}

\begin{proof}
C'est une conséquence de la proposition \ref{prop:nonadm} et de \cite[Prop. $4.5$]{EmCoh}.
\end{proof}

\begin{rema}
Ce dernier résultat est déjà bien connu. Si $F$ est une extension non ramifiée de $\mathbb{Q}_p$ et $p>2$, ce résultat est mentionné dans \cite[Rem. $4.2.6$]{Brmodp} et prouvé en détail dans \cite{MorraIwahori} pour $p\geq5$. Le cas totalement ramifié est traité dans \cite{Scheinsupersing} pour $p$ impair.
\end{rema}

\subsection{Preuve du théorème \ref{theo:principal}}

Supposons désormais $[\b:\mathbb{Q}_p]=2$. Soit $\sigma$ une représentation lisse irréductible de $\K$.\\

Soit $\phi_2$ l'endomorphisme de $A$ induit par $\alpha^2$. C'est un endomorphisme plat de $A$, égal au carré de $\phi$. Soient $I^p(\sigma)=\bigoplus_{m\in\mathbb{Z}}R_{2k}(\sigma)\subset\ind_\K^G(\sigma)$ et $I^{imp}(\sigma)=\bigoplus_{m\in\mathbb{Z}}R_{2k+1}(\sigma)$. On pose, pour $i\geq0$, $I^p_{\geq i}(\sigma)=I^p(\sigma)\cap I_{\geq i}(\sigma)$ et $I^{imp}_{\geq i}(\sigma)=I_{\geq i}(\sigma)\cap I^{imp}(\sigma)$. On a alors $T(I_{\geq1}^{imp}(\sigma))\subset I_{\geq0}^p(\sigma)$ et $T(I_{\geq2}^p(\sigma))\subset I_{\geq1}^{imp}(\sigma)$, de sorte que $I_{\geq0}^p(\sigma)/T(I_{\geq1}^{imp}(\sigma))$ est un sous-$A[X]_{\phi_2}$-module de $I_{\geq0}(\sigma)/T(I_{\geq1}(\sigma))$ et même facteur direct en tant que $A$-module. Posons dans la suite $L(\sigma)=I_{\geq0}^p(\sigma)/T(I_{\geq1}^{imp}(\sigma))$. Comme précédemment, on a un isomorphisme de $A[X]_{\phi_2}$-modules $A[X]_{\phi_2}\otimes_A\sigma \simeq I_{\geq0}^p(\sigma)$.\\
Remarquons que la preuve de la proposition \ref{prop:nonadm} s'adapte \emph{mutatis mutandis} au $A[X]_{\phi_2}$-module $L(\sigma)$ pour prouver que $\tor_2^A(\C,L(\sigma))$ n'est pas de $\C[X]$-torsion.

\begin{prop}
On a $\tor_0^A(\C,L(\sigma))=0$. La $\C[X]$-torsion de $\tor_2^A(\C,L(\sigma))$ est isomorphe à $\C[X]/(X)$ et coïncide avec l'image de $\tor_2^A(\C,\sigma)$ obtenue à partir du morphisme $\sigma\hookrightarrow I_{\geq0}^p(\sigma)\twoheadrightarrow L(\sigma)$, la première flèche étant $v\mapsto [1,v]$.
\end{prop}

\begin{proof}
Le lemme \ref{lemm:surjinj} montre que $\tor_0^A(T)=\tor_0^A(T_-)$ et que $\tor_0^A(\C,\sigma)\subset\tor_0^A(\C,I_{\geq0}^p(\sigma))$ est dans l'image de $\tor_0^A(T_-)$. Comme $\tor_0^A(\C,\sigma)$ engendre le $\C[X]$-module $\tor_0^A(\C,I_{\geq0}^p(\sigma))$, l'application $\tor_0^A(T)$ est une surjection de $\tor_0^A(\C,I_{\geq1}^{imp}(\sigma))$ sur $\tor_0^A(\C,I_{\geq0}^p(\sigma))$, ce qui prouve que $\tor_0^A(\C,L(\sigma))=0$.\\
De même, le lemme \ref{lemm:surjinj} montre que $\tor_2^A(T)=\tor_2^A(T_+)$. De plus $\tor_2^A(\C,R_2(\sigma))$ est l'image de $\tor_2^A(\C,R_1(\sigma))$ par $\tor_2^A(T_+)$.  Par ailleurs, $\tor_2^A(\C,I_{\geq0}^p(\sigma))$ est isomorphe à $\C[X]\otimes_\C\tor_2^A(\C,\sigma)$, et l'image de $\tor_2^A(\C,I_{\geq2}^p(\sigma))$ dans cet espace est exactement $X\C[X]\otimes_\C\tor_2^A(\C,\sigma)$ et est engendrée, comme $\C[X]$-module, par $\tor_2^A(\C,R_2(\sigma))$. Comme $T_+(I_{\geq1}^{imp}(\sigma))\subset I_{\geq2}^p(\sigma)$ on voit au final que l'image de $\tor_2^A(\C,I_{\geq1}^{imp}(\sigma))\rightarrow\tor_2^A(\C,I_{\geq0}^p(\sigma))$ est exactement $X\C[X]\otimes_\C\tor_2^A(\C,\sigma)$. On a alors une suite exacte longue
\begin{multline*}
0\rightarrow\tor_2^A(\C,I_{\geq1}^{imp}(\sigma))\rightarrow\tor_2^A(\C,I_{\geq0}^p(\sigma))\rightarrow\tor_2^A(\C,L(\sigma))\\\rightarrow\tor_1^A(\C,I_{\geq1}^{imp}(\sigma))\rightarrow\tor_1^A(\C,I_{\geq0}^p(\sigma))
\end{multline*}
et on conclut en remarquant que $\tor_1^A(\C,I_{\geq1}^{imp}(\sigma))\simeq\C[X]\otimes_\C\tor_1^A(\C,R_1(\sigma))$ est un $\C[X]$-module libre de type fini, donc le noyau de la flèche de droite est libre de type fini. Il est maintenant clair que $\tor_2^A(\C,L(\sigma))_{tors}\simeq\C[X]/(X)$ et que cette torsion est l'image de $\tor_2^A(\C,\sigma)\subset\tor_2^A(\C,I_{\geq0}^p(\sigma))$.
\end{proof}

\begin{theo}\label{theo:nonpf}
Soit $\pi$ une représentation lisse irréductible de $G$ ayant un caractère central. Si $\pi$ est supersingulière, alors elle n'est pas de présentation finie.
\end{theo}

\begin{proof}
Soit $\pi$ une représentation lisse irréductible supersingulière et ayant un caractère central. Par définition, il existe une représentation irréductible $\sigma$ de $\K$ et une surjection $G$-équivariante $\ind_\K^G(\sigma)/T\twoheadrightarrow\pi$. Comme $\sigma$ engendre la représentation $\ind_\K^G(\sigma)$, l'application composée $\sigma\hookrightarrow\ind_\K^G(\sigma)\twoheadrightarrow\pi$ est injective. Soit $I^+(\pi,\sigma)$ le sous-$A[X]_\phi$-module de $\pi$ engendré par $\sigma$, et $M(\pi,\sigma)$ le sous-$A[X]_{\phi_2}$-module engendré par $\sigma$. On a ainsi une surjection $L(\sigma)\twoheadrightarrow M(\pi,\sigma)$. Notons $N(\pi,\sigma)$ son noyau. Par ailleurs, $M(\pi,\sigma)$ est stable sous l'action du sous-groupe $H$, c'est donc un $\C[[\widetilde{H}]][X]_{\phi_2}$-module.\\
Comme la composée $\sigma\rightarrow L(\sigma)\rightarrow M(\pi,\sigma)$ est injective, il en est de même de l'application $\tor_2^A(\C,\sigma)\rightarrow\tor_2^A(\C,M(\pi,\sigma))$. Autrement dit, $\tor_2^A(\C,N(\pi,\sigma))\cap \tor_2^A(\C,L(\sigma))_{tors}=0$. Ainsi $\tor_2^A(\C,N(\pi,\sigma))$ est un $\C[X]$-module sans torsion. Comme on a par ailleurs $h_0(L(\sigma))=0$ et que $L(\sigma)$ n'est pas admissible puisque $\tor_2^A(\C,L(\sigma))$ n'est pas de torsion, le corollaire \ref{coro:nonadm} implique que $M(\pi, \sigma)$ n'est pas admissible en tant que $A$-module. Or $M(\pi,\sigma)\subset I^+(\pi,\sigma)$, donc d'après \cite[Thm. $1.3$]{HuDiag}, $\pi$ n'est pas de présentation finie.
\end{proof}

\def\cprime{$'$} \def\cprime{$'$} \def\cprime{$'$} \def\cprime{$'$}
\providecommand{\bysame}{\leavevmode ---\ }
\providecommand{\og}{``}
\providecommand{\fg}{''}
\providecommand{\smfandname}{\&}
\providecommand{\smfedsname}{\'eds.}
\providecommand{\smfedname}{\'ed.}
\providecommand{\smfmastersthesisname}{M\'emoire}
\providecommand{\smfphdthesisname}{Th\`ese}

\end{document}